\documentclass{article}

\usepackage[a4paper]{geometry}
\usepackage{amsmath,amssymb,amsthm,bm}
\usepackage{algorithm,algorithmic}
\usepackage{multirow}
\usepackage{graphicx}
\usepackage[colorlinks=false,pdfborder={0 0 0}]{hyperref}
\usepackage{authblk}

\graphicspath{{fig/}}

\newcommand*{\urf}{\bm{u}}

\newcommand{\herm}{^*}
\newcommand{\iherm}{^{-*}}
\newcommand{\fro}{\mathsf{F}}
\newcommand*{\abs}[1]{\lvert#1\rvert}
\newcommand*{\norm}[1]{\lVert#1\rVert}

\newcommand*{\hW}{\hat{W}}
\newcommand*{\hT}{\hat{T}}
\newcommand*{\hH}{\hat{H}}
\newcommand*\set[1]{\left\lbrace#1\right\rbrace}

\DeclareMathOperator{\diag}{diag}

\DeclareMathOperator{\rank}{rank}
\DeclareMathOperator{\range}{Range}
\DeclareMathOperator{\fl}{f{}l}

\DeclareMathOperator{\sign}{sign}

\renewcommand{\Re}{\mathop{\mathrm{Re}}}

\newenvironment{keywords}{\medskip\textbf{Keywords:}}{}
\newenvironment{AMS}{\medskip\textbf{AMS subject classifications (2020).}}{}

\theoremstyle{plain}
\newtheorem{theorem}{Theorem}
\newtheorem{lemma}{Lemma}

\theoremstyle{definition}
\newtheorem{assumption}{Assumption}
\newtheorem{example}{Example}

\theoremstyle{remark}
\newtheorem{remark}{Remark}

\title{On Two-Stage Householder Orthogonalization}
\author[1]{Zhuang-Ao He}
\author[1,2]{Meiyue Shao}
\affil[1]{School of Data Science, Fudan University, Shanghai 200433, China}
\affil[2]{MOE Key Laboratory for Computational Physical Sciences, Fudan
University, Shanghai 200433, China}
\date{\today}

\begin{document}

\maketitle

\begin{abstract}
Two-stage orthogonalization is essential in numerical algorithms such as
Krylov subspace methods.
For this task we need to orthogonalize a matrix \(A\) against another matrix
\(V\) with orthonormal columns.
A common approach is to employ the block Gram--Schmidt algorithm.
However, its stability largely depends on the condition number of \([V,A]\).
While performing a Householder orthogonalization on \([V,A]\) is
unconditionally stable, it does not utilize the knowledge that~\(V\) has
orthonormal columns.
To address these issues, we propose a two-stage Householder orthogonalization
algorithm based on the generalized Householder transformation.
Instead of explicitly orthogonalizing the entire \(V\), our algorithm only
needs to orthogonalizes a square submatrix of \(V\).
Theoretical analysis and numerical experiments demonstrate that our method is
also unconditionally stable.



\end{abstract}

\section{Introduction}
\label{sec:introduction}
Let \(V\in\mathbb{C}^{n\times k_0}\) have orthonormal columns (i.e.,
\(V\herm V=I_{k_0}\)), and \(A\in\mathbb{C}^{n\times k}\), where
\(k_0+k\ll n\).
In this paper, we study the two-stage orthogonalization problem---finding a
matrix \(Q\in\mathbb{C}^{n\times k}\) with orthonormal columns such that
\(\range([V,A])\subseteq\range([V,Q])\) and \(\range(Q)\perp\range(V)\).
This problem naturally arises in the block Krylov subspace methods for solving
linear systems~\cite{BDJ2006,XAD2024} and eigenvalue problems~\cite{DSYG2018,HL2006,
Knyazev2001}.
Theoretically, the problem is easy to solve because we can simply compute the
QR factorization of \([V,A]\).
Because \(V\) is already known to be orthonormal, we shall make use of this
knowledge to design a fast and stable approach.

Since \(V\) is orthonormal, we can first perform an \emph{inter-block
orthogonalization} step to orthogonalize \(A\) against \(V\), and then
perform an \emph{intra-block orthogonalization} step to compute the QR
factorization of the remaining \(A\).
This framework is called the \emph{BCGS framework}~\cite{BS2013,CLR2021,
CLRT2022}, which reads
\begin{subequations}
\label{eq:BQR}
\begin{align}
A&\gets A-VV\herm A,\label{eq:BQR-1}\\
A&=QR.\label{eq:BQR-2}
\end{align}
\end{subequations}
There are many choices of the intra-block orthogonalization
step~\eqref{eq:BQR-2}, including Householder-QR~\cite{DGHL2012,
Householder1958}, Givens-QR~\cite{Givens1958}, Gram--Schmidt
process~\cite{LBG2013}, Cholesky-QR~\cite{FKNYY2020}, and also
SVQB~\cite{SW2002} if \(R\) does not need to be upper triangular.

However, the problem is not as trivial as it may appear.
In practice, the input \([V,A]\) is sometimes very ill-conditioned or even
rank deficient~\cite{DSYG2018,HL2006}.
In this case the BCGS framework can fail to produce the desired orthogonal
basis in the presence of rounding errors.
For example, let us consider
\begin{align}
\label{eq:badbcg}
V=\frac12\begin{bmatrix}
\sqrt2 & \sqrt2 \\
-\sqrt2 & \sqrt2 \\
0 & 0 \\
0 & 0
\end{bmatrix},
\qquad
A=\begin{bmatrix}
1&1\\
1&1\\
10^{-30}&0\\0&10^{-30}
\end{bmatrix}.
\end{align}
The output of the BCGS framework (using almost any algorithm for the
intra-block orthogonalization step~\eqref{eq:BQR-2}) yields
\[
\bigl\lVert[V,Q]\herm[V,Q]-I_4\bigr\rVert_2\approx1.
\]
We remark that reorthogonalization does \emph{not} always produce satisfactory
results; see Table~\ref{tab:reorth}.
Therefore, the BCGS framework, even with reorthogonalization, is not always
numerically reliable.
In addition, when taking into account the cost of reorthogonalization, BCGS
can be more expensive than the most naive approach---directly computing the QR
factorization of \([V,A]\).

\begin{table}[!tb]
\centering
\caption{Loss of orthogonality under various settings (under double precision).}
\label{tab:reorth}
\begin{tabular}{cc}
\hline
Procedure & \(\vphantom{\Big|}\bigl\lVert[V,Q]\herm[V,Q]-I_4\bigr\rVert_2\) \\
\hline
\eqref{eq:BQR-1}, \eqref{eq:BQR-2} &
\(1.0\times10^0\) \\
\eqref{eq:BQR-1}, \eqref{eq:BQR-1}, \eqref{eq:BQR-2} &
\(9.8\times10^{-2}\) \\
\eqref{eq:BQR-1}, \eqref{eq:BQR-2}, \eqref{eq:BQR-1}, \eqref{eq:BQR-2} &
\(7.0\times10^{-2}\) \\
\eqref{eq:BQR-1}, \eqref{eq:BQR-1}, \eqref{eq:BQR-1}, \eqref{eq:BQR-2} &
\(2.2\times10^{-16}\) \\
\eqref{eq:BQR-1}, \eqref{eq:BQR-2}, \eqref{eq:BQR-1}, \eqref{eq:BQR-2},
\eqref{eq:BQR-1}, \eqref{eq:BQR-2} &
\(2.2\times10^{-16}\) \\
\hline
\end{tabular}
\end{table}

In this paper, we propose an algorithm based on the generalized Householder
transformation~\cite{BJ1974,Kaufman1987,SP1988} to perform the inter-block
orthogonalization~\eqref{eq:BQR-1}.
As for the intra-block orthogonalization step~\eqref{eq:BQR-2}, we mainly
focus on Householder-QR, which can achieve perfect orthogonality even for
ill-conditioned or rank deficient inputs.

The rest of this paper is organized as follows.
In Section~\ref{sec:existing}, we summarize some existing approaches to our
two-stage orthogonalization task.
In Section~\ref{sec:block}, we propose a two-stage Householder-QR algorithm
with three algorithmic variants.
A rounding error analysis is provided in Section~\ref{sec:rounding}.
In Section~\ref{sec:extensions}, we discuss how to extend our algorithm to the
context of non-standard inner products.
Numerical experiments in Section~\ref{sec:experiments} demonstrate the
accuracy and efficiency of our algorithm.

Throughout the paper, we adopt the MATLAB colon notation to describe
submatrices:
\(A_{i:j,k:l}\) is the submatrix of matrix \(A\) with row indices \(i\),
\(\dotsc\), \(j\) and column indices \(k\), \(\dotsc\), \(l\).
When the indices are omitted in the colon notation (e.g., \(A_{i:j,:}\) or
\(A_{:,k:l}\)), all indices in the corresponding dimension are included.

\section{Existing approaches}
\label{sec:existing}

\subsection{BCGS}
The BCGS framework as shown in~\eqref{eq:BQR} is a natural way to make use of
the orthogonality of~\(V\).
Typically, intra-block orthogonality is relatively straightforward to maintain,
while inter-block orthogonalization poses greater challenges.
To enhance orthogonality, reorthogonalization is often necessary.
Stathopoulos and Wu proposed the GS-SVQB method in~\cite{SW2002}, and Barlow
and Smoktunowicz introduced the BCGS2 method in~\cite{BS2013}.
Both methods improve the inter-block orthogonality by iterating inter-block
and intra-block orthogonalization steps.
Stewart proposed a method in~\cite{Stewart2008} that incorporates
randomization to tackle orthogonality faults during the orthogonalization
process.
However, Stewart's method is not suitable for high performance computing
because it it requires a large number of orthogonalization steps performed
sequentially.

\subsection{BMGS}
The BMGS framework~\cite{Barlow2019,CLRT2022} is an alternative to the BCGS
framework.
The matrix \(V\) is partitioned into \(V=[V_1,\dotsc,V_p]\).
Then the inter-block orthogonalization step~\eqref{eq:BQR-1} in BCGS is
replaced by
\[
A\gets(I-V_pV_p\herm)\cdots(I-V_1V_1\herm)A.
\]
The intra-block orthogonalization step remains the same as BCGS.
The motivation behind BMGS is the superior numerical stability of MGS compared
to CGS.

Due to the increased number of synchronization points compared to BCGS, BMGS
has a larger communication cost.
To address this issue, several low-synchronization variants of BMGS have been
proposed in~\cite{Barlow2019,CLRT2022}.
The basic idea is to utilize a WY-like representation~\cite{SV1989}, thereby
converting the process into
\[
A\gets(I-VTV\herm)A,
\]
where \(T\) is a lower triangular matrix associated with \(V\).
\footnote{In exact arithmetic, we have \(T=I\).}

Although BMGS is numerically more stable than BCGS, the stability also relies
on reorthogonalization, and can still lose orthogonality for extremely
ill-conditioned inputs (e.g., the example in~\eqref{eq:badbcg}).

\subsection{Cholesky-QR}
The Cholesky-QR algorithm can also be used to orthogonalize \([V,A]\) while
still making use of the orthogonality of \(V\);
see, e.g., \cite{CLR2021}.
Suppose that
\[
\begin{bmatrix}V,A\end{bmatrix}\herm\begin{bmatrix}V,A\end{bmatrix}=
\begin{bmatrix}
I_{k_0} &V\herm A\\ A\herm V& A\herm A
\end{bmatrix}=
\begin{bmatrix}
I_{k_0} &0\\ A\herm V& R\herm
\end{bmatrix}
\begin{bmatrix}
I_{k_0} &V\herm A\\ 0& R
\end{bmatrix},
\]
where \(R\herm\) is the Cholesky factor of \(A\herm A-(A\herm V)(V\herm A)\).
Then
\[
Q=(A-VV\herm A)R^{-1}.
\]
Reorthogonalization is also required in order to maintain the inter-block
orthogonality.
However, Cholesky-QR is susceptible to break down when applied to ill-conditioned matrices,
such as the example in \eqref{eq:badbcg}.
\section{Two-stage Householder-QR}
\label{sec:block}

A naive approach to achieving inter-block orthogonality would be to perform
Householder-QR on \([V,A]\), which would waste the knowledge that
\(V\herm V=I_{k_0}\).
In the following we discuss how to exploit this information to perform a
two-stage Householder-QR orthogonalization on \([V,A]\).
We shall demonstrate how some old techniques proposed in~\cite{BJ1974,
Kaufman1987,SP1988} decades ago can be used to solve this problem.

\subsection{Generalized Householder transformation}
Let \(X\), \(Y\in\mathbb C^{n\times m}\) such that \(X\herm X=Y\herm Y\).
Suppose that \(T=W\herm X\) is nonsingular, where \(W=X-Y\).
Then
\begin{equation}
\label{eq:generalized-Householder}
H=I_n-WT^{-1}W\herm
\end{equation}
is called a \emph{generalized Householder transformation}~\cite{BJ1974,
Kaufman1987,SP1988}.
It can be verified that \(H\herm H=I_n\) and \(HX=Y\).
Moreover, it is shown in~\cite{SB1995} that for every unitary matrix
\(H\in\mathbb C^{n\times n}\), there exist \(W\in\mathbb C^{n\times m}\) and
\(T\in\mathbb C^{m\times m}\) such that~\eqref{eq:generalized-Householder}
holds, where \(m=\rank(H-I_n)\).
Equation~\eqref{eq:generalized-Householder} is known as the
\emph{basis--kernel representation} of \(H\).

In order to orthogonalize \([V,A]\), the first stage is to perform \(k_0\)
Householder transformations to map \([I_{k_0},0]\herm\) onto \(V\).
These Householder transformations can be replaced by one generalized
Householder transformation with \(X=[I_{k_0},0]\herm\) and \(Y=V\).
In fact, the choice of \(X\) can be relaxed to \(X=[P\herm,0]\herm\), where
\(P\in\mathbb C^{k_0\times k_0}\) is an arbitrary unitary matrix.
Then the generalized Householder transformation becomes a unitary matrix of
the form \(H=[VP\herm,V_{\perp}]\) such that
\[
H\begin{bmatrix} P \\ 0 \end{bmatrix}=V.
\]

In the second stage, we first formulate
\[
H\herm A=\begin{bmatrix} PV\herm A \\ V_{\perp}\herm A \end{bmatrix}.
\]
Because \(H\herm A\) can be computed using the basis--kernel representation of
\(H\), the matrix \(V_\perp\), which is typically very large, does \emph{not}
need to be explicitly formed here.
Then we compute a smaller QR factorization \(V_{\perp}\herm A=\underline{Q}R\).
Let \(S=V\herm A\) and
\[
Q=H\begin{bmatrix} 0 \\ ~\underline{Q}~\end{bmatrix}.
\]
Then
\begin{equation}
\label{eq:vavq}
[V,A]=H\begin{bmatrix} P & 0~ \\ 0 & \underline{Q}~\end{bmatrix}
\cdot\begin{bmatrix} I_{k_0} & S \\ 0 & R \end{bmatrix}
=[V,Q]\begin{bmatrix} I_{k_0} & S \\ 0 & R \end{bmatrix}
\end{equation}
is the desired QR factorization of the input.
This two-stage Householder-QR orthogonalization process is summarized as
Algorithm~\ref{alg:main}.

\begin{algorithm}[!tb]
\caption{Two-stage Householder-QR.}
\label{alg:main}
\begin{algorithmic}[1]
\REQUIRE A matrix \(V\in\mathbb C^{n\times k_0}\) with orthonormal columns,
and another matrix \(A\in\mathbb C^{n\times k}\).
\ENSURE Matrices \(Q\in\mathbb{C}^{n\times k}\),
\(R\in\mathbb{C}^{k\times k}\), and \(S\in\mathbb{C}^{k_0\times k}\) that
satisfy~\eqref{eq:vavq} and \([V,Q]\herm[V,Q]=I_{k_0+k}\).

\STATE Choose a unitary matrix \(P\in\mathbb C^{k_0\times k_0}\).
\label{alg:line:init}
\STATE \(W\gets[P\herm,0]\herm-V\).
\STATE \(T\gets I_{k_0}-V_{1:k_0,:}\herm P\).
\STATE \(A\gets(I_n-WT\iherm W\herm) A\).
\STATE \(S\gets P\herm A_{1:k_0,:}\).
\STATE Compute the QR factorization \(A_{k_0+1:n,:}=\underline{Q}R\).
\STATE \(Q\gets(I_n-WT^{-1}W\herm)[0,\underline{Q}\herm]\herm\).
\end{algorithmic}
\end{algorithm}

\subsection{Choices of \(P\)}
\label{subsec:choices}
Line~\ref{alg:line:init} in Algorithm~\ref{alg:main} is vague since there are
different ways to choose the unitary matrix \(P\).
In the following, we discuss three particular choices.

\paragraph{First choice}
The simplest choice of \(P\) is a diagonal unitary matrix, e.g., \(P=I_{k_0}\).
As long as \(\lVert V_{1:k_0,:}\rVert_2<1\), the matrix
\(T=I_k-V_{1:k_0,:}\herm\) is nonsingular.
However, when \(\lVert V_{1:k_0,:}\rVert_2\) is close to one,
\(T\) can potentially be ill-conditioned.
In~\cite{BJ1974}, the authors suggest choosing
\[
P=-\diag\bigl\lbrace\sign(v_{1,1}),\dotsc,\sign(v_{k_0,k_0})\bigr\rbrace,
\]
where
\[
\sign(\mu)=\begin{cases}
1, & \text{if \(\Re\mu\geq0\),} \\
-1, & \text{if \(\Re\mu<0\).}
\end{cases}
\]
Unfortunately, such a choice of \(P\) cannot ensure a well-conditioned \(T\).

Noting that \(T=(P-V_{1:k_0,:})\herm P\), solving the linear system involving
\(T\) entails performing an LU factorization of \(P-V_{1:k_0,:}\).
A better choice proposed in~\cite{BDGJKN2015} is to determine \(P\) on-the-fly
when computing the LU factorization of \(P-V_{1:k_0,:}\), as is
illustrated in Algorithm~\ref{alg:mlu}.
Pivoting is unnecessary here since the magnitude of the diagonal entry never
drops below one.

\begin{algorithm}[!tb]
\caption{Modified LU factorization~\cite{BDGJKN2015}.}
\label{alg:mlu}
\begin{algorithmic}[1]
\REQUIRE A matrix \(Z\in\mathbb{C}^{k\times k}\) such that
\(\lVert Z\rVert_2\leq1\).
\ENSURE A unitary diagonal matrix \(P\in\mathbb C^{k\times k}\) and the LU
factorization of \(P-Z\).

\STATE \(P\gets I_k\), \(L\gets I_k\), \(U\gets0\)
\FOR{\(i=1\) to \(k\)}
\STATE \(P_{i,i}\gets-\sign(Z_{i,i})\).
\STATE \(U_{i,i}\gets P_{i,i}-Z_{i,i}\).
\STATE \(U_{i,i+1:k}\gets-Z_{i,i+1:k}\).
\STATE \(L_{i+1:k,i}\gets-Z_{i+1:k,i}U_{i,i}^{-1}\).
\STATE \(Z_{i+1:k,i+1:k}\gets Z_{i+1:k,i+1:k}+L_{i+1:k,i}U_{i,i+1:k}\).
\ENDFOR
\end{algorithmic}
\end{algorithm}

However, we remark that a generalized Householder transformation based on the
modified LU factorization is not always stable, because \(\kappa_2(U)\) can
be very large.
For instance, let
\begin{equation}
\label{eq:badmlu}
U=\begin{bmatrix}
1+\alpha/\sqrt{k_0} & -1/\sqrt{k_0} & \cdots & -1/\sqrt{k_0} \\
0 & 1+\alpha/\sqrt{k_0-1} & \cdots & -1/\sqrt{k_0-1} \\
\vdots & \vdots & \ddots & \vdots \\
0 & 0 & \cdots & 1+\alpha
\end{bmatrix}.
\end{equation}
When \(k_0=100\), \(\alpha=0.1\), we have
\(\kappa_2(U)\approx1.1\times10^{7}\).
In Appendix~\ref{sec:LU} we provide a proof to show that the matrix \(U\)
in~\eqref{eq:badmlu} is a feasible output of the modified LU factorization.
Example~\ref{ex:badmlu} in Section~\ref{sec:experiments} is also constructed
based on~\eqref{eq:badmlu}.

\paragraph{Second choice}
Suppose that the QR factorization of \(V_{1:k_0,:}\) is \(Q_1R_1\), where
\(R_1\) is an upper triangular matrix with nonnegative diagonal entries.
We can choose \(P=-Q_1\).
This is essentially the choice in~\cite{Kaufman1987}.
Then \(T=I_{k_0}+R_1\herm\) is a nonsingular lower triangular matrix that
satisfies \(\lVert T\rVert_2\leq2\) and \(T+T\herm=W\herm W\).
Partitioning \(W=[w_1,\dotsc,w_{k_0}]\) yields
\[
\lVert w_i\rVert_2^2=2T_{i,i}\geq2.
\]
Then \(0<(T^{-1})_{i,i}=T_{i,i}^{-1}\leq1\).
According to \cite[Theorem 12]{BS1994}, for \(i>j\) we have
\[
\bigl\lvert(T^{-1})_{i,j}\bigr\rvert
\leq\frac{4}{\lVert w_i\rVert_2\lVert w_j\rVert_2}
\leq2.
\]
Thus,
\[
\lVert T^{-1}\rVert_2
\leq\lVert T^{-1}\rVert_{\fro}
\leq\sqrt{2k_0^2-k_0}
<\sqrt2\,k_0,
\]
and \(\kappa_2(T)<2\sqrt2\,k_0\).
We conclude that the choice \(P=-Q_1\) always ensures a well-conditioned~\(T\).

\paragraph{Third choice}
If the polar decomposition \(V_{1:k_0,:}=Q_2M\) is computed,
we can choose \(P=-Q_2\) as suggested by~\cite{BJ1974,SP1988}.
Note that \(\lVert M\rVert_2\leq1\).
Then \(T=I+M\) is a positive definite matrix with \(\kappa_2(T)\leq2\).
Therefore, \(T^{-1}\) can be computed accurately (either explicitly or
implicitly by solving linear systems) using the Cholesky factorization.
This choice is better than the second choice in theory, because the upper bound
of \(\kappa_2(T)\) does not depend on \(k_0\).
However, in practice, computing the polar decomposition is in general more
expensive compared to computing the QR factorization as in the second choice.

\section{Rounding error analysis}
\label{sec:rounding}

In this section, we conduct a rounding error analysis to demonstrate that
Algorithm~\ref{alg:main} is \emph{unconditionally stable} (i.e., the stability
does not depend on \(\kappa_2([V,A])\).
We adopt the standard rounding model
\[
\fl(a\circ b)=(a\circ b)(1+\delta),
\qquad \abs{\delta}\le\urf,
\]
where \(\urf\) is the unit roundoff and \(\circ\in\set{+,-,\times,/}\).
Moreover, we define \(\gamma_n=(1-\urf)^{-n}-1\).
To begin with, we make some plausible assumptions.

\begin{assumption}
\label{assump:qr}
Assume that \(A\in\mathbb{C}^{m\times n}\).
For the QR factorization \(A=QR\), the computed factors \(\hat{Q}\) and
\(\hat{R}\) satisfy
\[
\lVert\hat{Q}\herm \hat{Q}-I_n\rVert_2\le\epsilon_{\mathrm{orth}}(m,n),
\qquad \lVert A-\hat{Q}\hat{R}\rVert_2
\le\epsilon_{\mathrm{res}}(m,n)\lVert A\rVert_2,
\]
where \(\epsilon_{\mathrm{orth}}(m,n)\) and \(\epsilon_{\mathrm{res}}(m,n)\)
are functions of \(\urf\), \(m\), and \(n\)
satisfying~\(\epsilon_{\mathrm{orth}}\),
\(\epsilon_{\mathrm{res}}\ll1\).
As a result, we have \(\lVert\hat{Q}\rVert_2\le\sqrt{2}\) and
\(\lVert\hat{R}\rVert_2\le\sqrt{2}\,\lVert A\rVert_2\).
\end{assumption}

\begin{assumption}
\label{assump:T}
We assume that the computed value \(\hT\) of \(T=(P-V_{1:k_0,:})\herm P\)
satisfies
\[
\lVert\Delta T\rVert_2=\lVert\hT-T\rVert_2\le\epsilon_T(k_0),
\]
where \(\epsilon_T(k_0)\) denotes a function of \(\urf\) and \(k_0\).
\end{assumption}

\begin{assumption}
\label{assump:sv}
Given \(A\in\mathbb{C}^{n\times n}\) and \(b\in\mathbb{C}^n\), the computed
solution \(\hat{x}\) of \(Ax=b\) satisfies
\[
(A+\Delta A)\hat{x}=b,
\qquad \lVert\Delta A\rVert_2\leq\epsilon_{\mathrm{sv}}(n)\lVert A\rVert_2,
\]
where \(\epsilon_{\mathrm{sv}}(n)\) denotes a function of \(\urf\)
and \(n\).
\end{assumption}

In Appendix~\ref{sec:assump}, we illustrate that these assumptions are valid
for a variety of algorithmic choices.
Rather than expressing the upper bounds as \(O(\urf)\) or concrete constants
like \(\gamma_n\), we express them dimension-dependent functions.
This allows us to trace and locate the source of rounding errors.

We first establish Theorem~\ref{thm:H}, which characterizes the quality of the
generalized Householder transformation if the condition \(X\herm X=Y\herm Y\)
is slightly violated.

\begin{theorem}
\label{thm:H}
Let \(X\), \(Y\in\mathbb{C}^{n\times k}\) and \(W=X-Y\).
Suppose that \(T=W\herm X\) is nonsingular and
\(\lVert X\herm X-Y\herm Y\rVert_2\le\delta\).
Then the matrix \(H=I_n-WT^{-1}W\herm\) satisfies
\begin{align*}
\lVert H\herm H-I_n\rVert_2\le\lVert WT\iherm\rVert_2^2\cdot\delta,
\qquad \lVert H\herm Y-I_n\rVert_2\le\lVert WT\iherm\rVert_2\cdot\delta,
\end{align*}
and \(HX=Y\).
In particular, if \(X\herm X=Y\herm Y\), then \(H\) is unitary.
\end{theorem}

\begin{proof}
We have
\begin{align*}
H\herm H&=I_n-WT\iherm\big(T+T\herm-W\herm W\big)T^{-1}W\herm.
\end{align*}
Hence
\begin{align*}
\lVert H\herm H-I_n\rVert_2&=\lVert WT\iherm\big(T+T\herm-W\herm W\big)T^{-1}W\herm\rVert_2\\
&=\lVert WT\iherm\big(X\herm X-Y\herm Y\big)T^{-1}W\herm\rVert_2\\
&\le\lVert WT\iherm\rVert_2^2\cdot\delta.
\end{align*}
The other bound follows that
\begin{align*}
\lVert H\herm Y-X\rVert_2&=\lVert Y-WT\iherm(-X\herm W+X\herm W+W\herm Y)-X\rVert_2\\
&=\lVert WT\iherm(X\herm X-Y\herm Y)\rVert_2\\
&\le\lVert WT\iherm\rVert_2\cdot\delta.
\end{align*}
It can be easily verified that \(HX=Y\).
If \(X\herm X=Y\herm Y\), then we can choose \(\delta=0\), and thus \(H\) is
unitary.
\end{proof}

\begin{remark}
In Algorithm~\ref{alg:main}, our choice of \(X\) and \(Y\) satisfies
\[
\lVert X\herm X-Y\herm Y\rVert_2
=\lVert P\herm P-V\herm V\rVert_2
=\lVert I_{k_0}-V\herm V\rVert_2.
\]
Therefore, the orthogonality of \(H\) mainly depends on the orthogonality of
the input \(V\).
However, we remark that \(\lVert HX-Y\rVert_2\), which is zero in exact
arithmetic, is independent of \(\lVert X\herm X-Y\herm Y\rVert_2\).
\end{remark}

It should be emphasized that \(\lVert WT^{-1}\rVert_2\) and
\(\lVert WT\iherm\rVert_2\) may be significantly smaller than
\(\lVert T^{-1}\rVert_2\) when \(T\) is ill-conditioned.
We have the following lemma.

\begin{lemma}
\label{lem:WT}
Suppose \(W\in\mathbb{C}^{n\times k}\), \(T\in\mathbb{C}^{k\times k}\).
If \(\lVert W\herm W-T-T\herm\rVert_2\le\delta\), then
\[
\max\set{\lVert WT^{-1}\rVert_2,\lVert WT\iherm\rVert_2}
\leq\sqrt{\lVert T^{-1}\rVert_2^2\delta+2\lVert T^{-1}\rVert_2}.
\]
\end{lemma}

\begin{proof}
Let \(E=W\herm W-T-T\herm\).
Then
\[
\lVert WT^{-1}\rVert_2
=\sqrt{\lVert T\iherm W\herm WT^{-1}\rVert_2}
=\sqrt{\lVert T\iherm ET^{-1}+T^{-1}+T\iherm\rVert_2}
\leq\sqrt{\lVert T^{-1}\rVert_2^2\delta+2\lVert T^{-1}\rVert_2}.
\]
Similarly, we have
\(\lVert WT\iherm\rVert_2
\leq\sqrt{\lVert T^{-1}\rVert_2^2\delta+2\lVert T^{-1}\rVert_2}\).
\end{proof}

Our choice of \(X\) and \(Y\) in Algorithm~\ref{alg:main} satisfies
\[
\lVert X\herm X-Y\herm Y\rVert_2=\lVert W\herm W-T-T\herm\rVert_2.
\]
In the subsequent analysis, backward errors will be propagated on \(W\) and
\(T\).
The following theorem addresses the case where both \(W\) and \(T\) are
subject to perturbations.

\begin{theorem}
\label{thm:hH}
Under the assumptions of Theorem~\ref{thm:H},
suppose \(\hW=W+\Delta W\) and \(\hT=T+\Delta T\),
where \(\lVert\Delta W\rVert_2\le\delta_W\), \(\lVert\Delta T\rVert_2\le\delta_T\) and \(\delta_T\lVert T^{-1}\rVert_2\le1/2\).
Let
\[
\epsilon_1=\delta+2\delta_T+2\delta_W\norm{W}_2+\delta_W^2,\qquad
\eta=2\sqrt{\lVert T^{-1}\rVert_2^2\epsilon_1+\lVert T^{-1}\rVert_2}.
\]
Then \(\hH=I_n-\hW\hT^{-1}\hW\herm\) satisfies
\begin{align*}
\lVert\hH-H\rVert_2&\le(\delta_T+\delta_W^2) \eta^2+2\delta_W\cdot\eta,\\
\lVert\hH X-Y\rVert_2&\le\delta_W+
(\delta_T+\delta_W\lVert X\rVert_2)\eta,\\
\lVert\hH\herm Y-X\rVert_2&\le\delta_W
+(\delta+\delta_T+\delta_W\lVert Y\rVert_2)\eta.
\end{align*}
\end{theorem}

\begin{proof}
We have
\begin{align*}
\lVert\hH-H\rVert_2&=\lVert\hW\hT^{-1}\hW\herm-WT^{-1}W\herm\rVert_2\\
&\le\lVert\hW\hT^{-1}\hW\herm-\hW T^{-1}\hW\herm\rVert_2
+\lVert\hW T^{-1}\hW\herm-WT^{-1}W\herm\rVert_2\\
&=\lVert\hW\hT^{-1}\Delta TT^{-1}\hW\herm\rVert_2
+\lVert WT^{-1}\Delta W\herm+\Delta W T^{-1}W\herm+\Delta WT^{-1}\Delta W\herm\rVert_2\\
&\le\delta_T\lVert\hW \hT^{-1}\rVert_2\lVert\hW T\iherm\rVert_2
+\delta_W(\lVert W T^{-1}\rVert_2+\lVert W T\iherm\rVert_2+\delta_W\lVert T^{-1}\rVert_2).
\end{align*}
Note that \(\lVert\hW\herm\hW-\hT-\hT\herm\rVert_2\le \epsilon_1\)
and \(\lVert\hW\herm\hW-T-T\herm\rVert_2\le \epsilon_1\).
By Lemma 2.3.3 in~\cite[Section~2.3.4]{GV2013}, we obtain
\[
\lVert\hT^{-1}\rVert_2\le\frac{\lVert T^{-1}\rVert_2}
{1-\lVert T^{-1}\Delta T\rVert_2}\le2\lVert T^{-1}\rVert_2.
\]
Using Lemma~\ref{lem:WT}, we have
\[
\max\set{\lVert\hW\hT^{-1}\rVert_2,\lVert\hW\hT\iherm\rVert_2,
\lVert\hW T^{-1}\rVert_2,\lVert\hW T\iherm\rVert_2}\leq\eta,
\]
and then
\[
\lVert\hH-H\rVert_2\le(\delta_T+\delta_W^2) \eta^2+2\delta_W\cdot\eta.
\]
Note that
\begin{align*}
\hH X&=X-\hW\hT^{-1}(W\herm X+\Delta T)
+\hW\hT^{-1}(\Delta T-\Delta W\herm X)\\
&=Y-\Delta W+\hW\hT^{-1}(\Delta T-\Delta W\herm X).
\end{align*}
Thus
\begin{align*}
\lVert\hH X-Y\rVert_2\le\delta_W+
\lVert\hW\hT^{-1}\rVert_2(\delta_T+\delta_W\lVert X\rVert_2).
\end{align*}
Additionally, it holds that
\begin{align*}
\label{eqn:H*Y}
\hH\herm Y&=Y+\hW\hT\iherm(X\herm W+\Delta T\herm)-\hW\hT\iherm(X\herm W+\Delta T\herm+\hW\herm Y)\\
&=X+\Delta W-\hW\hT\iherm(X\herm X-Y\herm Y+\Delta T\herm+\Delta W\herm Y).
\end{align*}
We can verify that
\[
\lVert\hH\herm Y-X\rVert_2\le\delta_W+\lVert\hW\hT\iherm\rVert_2(\delta+\delta_T+\delta_W\lVert Y\rVert_2).
\]
This completes the proof.
\end{proof}

The following theorem establishes the stability of Algorithm~\ref{alg:main}.
The proof is somewhat lengthy and we leave it in Appendix~\ref{sec:proof}.

\begin{theorem}
\label{thm:orth}
Suppose that \(\delta_1=\lVert V\herm V-I_{k_0}\rVert_2\ll1\)
and \(\delta_2=\lVert P\herm P-I_{k_0}\rVert_2\ll1\).
Denote the computed value of each algorithmic variable be the original variable wearing a hat.
We define
\begin{align*}
\delta_T&=\epsilon_T(k_0)+3\epsilon_{\mathrm{sv}}(k_0)
(1+\epsilon_T(k_0)),\\
\delta&=\delta_1+\delta_2,\\
\epsilon_1&=\delta+2\delta_T+16\sqrt{k_0}\gamma_{k_0+1}+8k_0\gamma_{k_0+1}^2,\\
\eta&=2\sqrt{\norm{T^{-1}}_2^2\epsilon_1+\norm{T^{-1}}_2},\\
\epsilon_2&=(\delta_T+8k_0\gamma_{k_0+1}^2) \eta^2+
4\sqrt{2k_0}\gamma_{k_0+1}\cdot\eta,\\
\epsilon_3&=2\sqrt{2k_0}\gamma_{n+2k_0+1}\cdot\eta
+\urf(\sqrt{1+\delta \eta^2}+\epsilon_2),\\
\epsilon_4&=\epsilon_2+\epsilon_3.
\end{align*}
If \(\delta_T\norm{T^{-1}}_2\le1/2\), then \(\hat{Q}\) satisfies
\begin{align*}
\lVert\hat{Q}\herm \hat{Q}-I\rVert_2
&\le\epsilon_{\mathrm{orth}}(n-k_0,k)
+2(\delta \eta^2+2\sqrt{k(1+\delta \eta^2)}\,\epsilon_4+k\epsilon_4^2),\\
\lVert V\herm \hat{Q}\rVert_\fro&\le \sqrt{2k}\,(2\sqrt{2k_0}\gamma_{k_0+1}+
\epsilon_1\eta+\sqrt{2}\epsilon_3),
\end{align*}
and the residual \(A-\hat{Q}\hat{R}-V\hat{S}\) satisfies
\begin{align*}
\frac{\lVert A-\hat{Q}\hat{R}-V\hat{S}\rVert_2}{\lVert A\rVert_2}
&\le\delta\eta^2+(\sqrt{1+\delta \eta^2}+\sqrt{k}\,\epsilon_4)^2
(\epsilon_{\mathrm{res}}(n-k_0,k)+\delta_2+3\sqrt{k}\,\epsilon_4+2k_0\gamma_{k_0}).
\end{align*}
\end{theorem}

\begin{remark}
The assumption \(\delta_T\lVert T^{-1}\rVert_2\leq1/2\) in
Theorem~\ref{thm:orth} is plausible.
In fact, \(\delta_T\lVert T^{-1}\rVert_2\ll1\) automatically holds as long as
\(T\) is reasonably well-conditioned, as ensured by the second and third choices of
\(P\) in Section~\ref{subsec:choices}.
\end{remark}

\begin{remark}
When \(\kappa_2(T)\urf\ll1\), we have \(\eta=O(\sqrt{\kappa_2(T)}\,)\).
Thus
\begin{align*}
\lVert\hat{Q}\herm \hat{Q}-I\rVert_2&=O(\kappa_2(T)\urf),\\
\lVert V\herm \hat{Q}\rVert_2&=O(\sqrt{\kappa_2(T)}\,\urf),\\
\frac{\lVert A-\hat{Q}\hat{R}-V\hat{S}\rVert_2}{\lVert A\rVert_2}
&=O(\kappa_2(T)\urf).
\end{align*}
We conclude that the stability of Algorithm~\ref{alg:main} does not depend on
\(\kappa_2([V,A])\).
The second and third choices of \(P\) yield unconditionally stable algorithms.
Even if the first choice is not always stable because \(\kappa_2(T)\) can be
large, the source of instability is \emph{not} the numerical rank deficiency
of \([V,A]\).
\end{remark}

\section{Extension to non-standard inner products}
\label{sec:extensions}

In~\cite{Shao2023}, a Householder-QR algorithm in the context of a
non-standard inner product is proposed.
Our algorithm can be extended to the non-standard inner product setting based
on the left-looking variant in~\cite{Shao2023}.

Let \(B\in\mathbb{C}^{n\times n}\) be Hermitian and positive definite.
The \(B\)-inner product is defined as
\[
\langle x,y\rangle_B=y\herm Bx.
\]
We can obtain the following result similar to Theorem~\ref{thm:H}.

\begin{theorem}
\label{thm:bgh}
Let \(X\), \(Y\in\mathbb{C}^{n\times k}\) and \(W=X-Y\).
Suppose that \(\norm{X\herm BX-Y\herm BY}_2\le\delta\) and \(T=W\herm BX\) are
nonsingular.
Then the matrix \(H=I_n-WT^{-1}W\herm B\) satisfies
\begin{align*}
\norm{H\herm BH-B}_2\le\norm{BWT\iherm}_2^2\cdot\delta
\end{align*}
and \(HX=Y\).
If \(X\herm BX=Y\herm BY\), then \(H\herm BH=B\) and
\(H^{-1}=I_n-WT\iherm W\herm B\).
\end{theorem}

Now we can derive the two-stage Householder-QR under the \(B\)-inner product.
Similar to~\cite{Shao2023}, we first construct an initial \(B\)-orthonormal
basis \(U=[u_1,\dotsc,u_{k_0+k}]\in\mathbb{C}^{n\times(k_0+k)}\), which can be
obtained, e.g., via the Cholesky factorization of a principal submatrix of
\(B\).
Then we apply the techniques detailed in Section~\ref{subsec:choices} to
choose a unitary \(P\) for improving numerical stability (based on the matrix
\(V\herm BU_{:,1:k_0}\)).
The first \(k_0\) columns of \(U\) is replaced by~\(\tilde{U}=U_{:,1:k_0}P\).
Let \(X=[\tilde{u}_1,\dots,\tilde{u}_{k_0}]\), \(Y=V\) and then \(H^{-1}A\)
can be written as
\[
H^{-1}A=H^{-1}(VV\herm BA+V_\perp V_\perp\herm BA)=\tilde{U}V\herm BA+\tilde{U}_\perp V_\perp\herm BA.
\]
The next step is to remove components contributed by \(\tilde{u}_1\),
\(\dotsc\), \(\tilde{u}_{k_0}\).
This can be accomplished by a Gram--Schmidt orthogonalization process.
Then we only need to orthogonalize \(\tilde{U}_\perp V_\perp\herm BA\) using
\(u_{k_0+1}\), \(\dotsc\), \(u_{k_0+k}\) as the initial \(B\)-orthonormal
basis.
The final step is to transform \(\tilde{U}_\perp\) to~\(V_\perp\), which can
be accomplished via multiplication by \(H\).
Similar to~\eqref{eq:vavq}, we have
\begin{equation}
\label{eq:vavqb}
[V,A]=H\,[U_{1:k_0,:}P,\,\underline{Q}\,]
\cdot\begin{bmatrix} I_{k_0} & S \\ 0 & R \end{bmatrix}
=[V,Q]\begin{bmatrix} I_{k_0} & S \\ 0 & R \end{bmatrix}.
\end{equation}

We summarize the algorithm in Algorithm~\ref{alg:main-B}.
Notice that, to ensure orthogonality in finite precision arithmetic, a
reorthogonalization step of the Householder vector \(w_i\) is recommended
in~\cite{Shao2023}.
The situation is similar here---it is highly recommended to reorthogonalize
\(w_i\) against
\([\tilde{u}_1,\dotsc,\tilde{u}_{k_0},u_{k_0+1},\dotsc,u_{k_0+i-1}]\).

\begin{algorithm}[!tb]
\caption{Two-stage Householder-QR (\(B\)-inner product).}
\label{alg:main-B}
\begin{algorithmic}[1]
\REQUIRE A positive definite matrix \(B\in\mathbb C^{n\times n}\),
two matrices \(V\in\mathbb C^{n\times k_0}\) and
\(U\in\mathbb C^{n\times(k_0+k)}\) with \(B\)-orthonormal columns,
and another matrix \(A\in\mathbb C^{n\times k}\).
\ENSURE Matrices \(Q\in\mathbb C^{n\times k}\), \(R\in\mathbb C^{k\times k}\),
and \(S\in\mathbb{C}^{k_0\times k}\) that satisfy \eqref{eq:vavqb}
and \([V,Q]\herm B[V,Q]=I_{k_0+k}\).

\STATE Choose a unitary matrix \(P\in\mathbb C^{k_0\times k_0}\).
\STATE \(U_{:,1:k_0}\gets U_{:,1:k_0}P\).
\STATE \(W\gets U_{:,1:k_0}-V\).
\STATE \(T\gets I_{k_0}-V\herm BU_{:,1:k_0}\).
\STATE \(A\gets(I_n-WT\iherm W\herm B) A\).
\STATE \(S\gets U_{:,1:k_0}\herm BA\).
\STATE \(A\gets A-U_{:,1:k_0}S\).
\STATE Compute the QR factorization \(A=\underline{Q}R\) using Algorithm 1 or 2 in \cite{Shao2023}.
\STATE \(Q\gets(I_n-WT^{-1}W\herm B)\underline{Q}\).
\end{algorithmic}
\end{algorithm}

\section{Numerical experiments}
\label{sec:experiments}
In this section we test our algorithms with several examples.
All experiments are performed under IEEE double precision arithmetic with
\(\urf=2^{-53}\approx1.1\times10^{-16}\) on a Linux server equipped with
two sixteen-core Intel Xeon Gold 6226R 2.90~GHz CPUs with 1024~GB of main
memory.

\subsection{Stability tests}
The computations are carried out using MATLAB 2023b in this subsection.
Random matrices are generated by \texttt{randn}.
We use the MATLAB function \texttt{qr} for Householder-QR.
In Algorithm~\ref{alg:main}, we use the three choices of \(P\) discussed in
Section~\ref{subsec:choices}, where a variant of QR
factorization~\cite{DHHR2009} is adopted to produce \(R\) with nonnegative
diagonal entries.
Additionally, we use SVD to compute the polar decomposition.

\begin{example}
In~\eqref{eq:badbcg}, we presented \(V\) and \(A\) for which the BCGS
framework struggles to compute~\(Q\).
Our algorithm with all three choices successfully yields \(\hat{Q}\) which
satisfies \(\bigl\lVert[V,\hat{Q}]\herm[V,\hat{Q}]-I_4\bigr\rVert_2
\approx2\urf\).
\end{example}

Algorithm~\ref{alg:main} can be extended to orthogonalize multiple blocks like
the algorithms in~\cite{CLRT2022}.
We present the complete procedure in Algorithm~\ref{alg:bh}.
This algorithm targets applications where matrices \(A_i\) are produced
sequentially, which represents the standard mode of operation in block Krylov
subspace methods.
Next, we compare Algorithm~\ref{alg:bh} with BCGS algorithms to demonstrate
the numerical stability advantages of our approach.

\begin{algorithm}[!tb]
\caption{Block Householder-QR.}
\label{alg:bh}
\begin{algorithmic}[1]
\REQUIRE \(A=[A_1,\dotsc,A_p]\in\mathbb{C}^{n\times k}\), \(A_i\in\mathbb{C}^{n\times k_i}\), \(i=1\), $\dots$, \(p\).
\ENSURE Matrices \(Q\in\mathbb{C}^{n\times k}\), \(R\in\mathbb{C}^{k\times k}\)
that satisfy \(Q\herm Q=I_k\) and \(A=QR\).
\STATE Compute the QR factorization \(A_1=Q_1R_1\).
\FOR{\(i=2\) \textbf{to} \(p\)}
\STATE Construct generalized Householder matrix \(H\) based on \([Q_1,\dotsc,Q_{i-1}]\).
\STATE \(X=A_i\).
\STATE \(X=H\herm X\).
\STATE Compute the QR factorization \(X_{(k_1+\cdots+k_{i-1}+1):n,:}=\underline{Q}_iR_i\).
\STATE \(Q_i=H\begin{bmatrix}0\\~\underline{Q}_i~\end{bmatrix}\).
\ENDFOR
\end{algorithmic}
\end{algorithm}

\begin{example}
\label{ex:concrete}
We test the stability of Algorithm~\ref{alg:bh} in multiple blocks
on two ill-conditioned matrices---\texttt{s-step} and \texttt{stewart\_extreme}
from~\cite{CLRT2022}.
The matrix \(A\in\mathbb{R}^{n\times pk}\), where \(n=10000\)
(number of rows), \(p=50\) (number of blocks), and \(k=10\) (columns per blocks).
The comparison with BCGS2 is shown in Table~\ref{tab:orbh},
where the labels `House' and `Chol', respectively, indicate that
Householder-QR and the shifted Cholesky-QR algorithm~\cite{FKNYY2020} are
employed for intra-block orthogonalization.
Our algorithms maintain orthogonality to machine precision, whereas BCGS2 may
completely loses orthogonality.

\begin{table}[!tb]
\centering
\caption{Losses of orthogonality and relative residuals for different
	algorithms on Example~\ref{ex:concrete}, where ``--'' indicates that the
	shifted Cholesky-QR algorithm fails at some step of the computation.}
\begin{tabular}{ccccc}
	\hline
	\multirow{2}{*}{Method}
	&\multicolumn{2}{c}{\texttt{s-step}}&
	\multicolumn{2}{c}{\texttt{stewart\_extreme}}\\
	&\(\norm{\hat{Q}\herm \hat{Q}-I}_2\)&\(\norm{A-\hat{Q}\hat{R}}_2/\norm{A}_2\)
	&\(\norm{\hat{Q}\herm \hat{Q}-I}_2\)&\(\norm{A-\hat{Q}\hat{R}}_2/\norm{A}_2\)\\
	\hline
	Alg.~\ref{alg:bh} (1st choice)&\(\hphantom{i}7.37\times 10^{-15}\)&\(2.10\times 10^{-15}\)&\(\hphantom{i}1.28\times 10^{-15}\)&\(7.74\times 10^{-16}\)\\
	\hspace{3.8pt}Alg.~\ref{alg:bh} (2nd choice)&\(\hphantom{i}1.02\times 10^{-14}\)&\(2.27\times 10^{-15}\)&\(\hphantom{i}1.13\times 10^{-15}\)&\(6.53\times 10^{-16}\)\\
	\hspace{2pt}Alg.~\ref{alg:bh} (3rd choice)&\(\hphantom{i}1.42\times 10^{-14}\)&\(2.61\times 10^{-15}\)&\(\hphantom{i}1.98\times 10^{-15}\)&\(1.35\times 10^{-15}\)\\
	BCGS2 (House)&\(4.20\times 10^1\hphantom{0}\)&\(7.62\times 10^{-15}\)&\(2.86\times 10^0\hphantom{0}\)&\(1.84\times 10^{-16}\)\\
	BCGS2 (Chol)&--&--&\(\hphantom{i}1.54\times 10^{-14}\)&\(1.33\times 10^{-16}\)\\
	\hline
\end{tabular}
\label{tab:orbh}
\end{table}
\end{example}

\begin{example}
\label{ex:concrete-B}
We also test the stability of Algorithm~\ref{alg:bh} with a non-standard inner
product.
We use the same test matrices as in Example~\ref{ex:concrete}, and generate a
reasonably well-conditioned Hermitian positive definite matrix \(B\) with
\(\kappa_2(B)=10^5\).
We compare them with BCGS2 and the results are shown
in Table~\ref{tab:orbhb}.
Similarly, our algorithms achieve far better orthogonality than BCGS2.

\begin{table}[!tb]
\centering
\caption{Losses of orthogonality and relative residuals for different
	algorithms with a non-standard inner product on Example~\ref{ex:concrete-B},
	where ``--'' indicates that the shifted Cholesky-QR algorithm fails at some
	step of the computation.}
\begin{tabular}{ccccc}
	\hline
	\multirow{2}{*}{Method}
	&\multicolumn{2}{c}{\texttt{s-step}}&
	\multicolumn{2}{c}{\texttt{stewart\_extreme}}\\
	&\(\norm{\hat{Q}\herm B\hat{Q}-I}_2\)&\(\norm{A-\hat{Q}\hat{R}}_2/\norm{A}_2\)
	&\(\norm{\hat{Q}\herm B\hat{Q}-I}_2\)&\(\norm{A-\hat{Q}\hat{R}}_2/\norm{A}_2\)\\
	\hline
	Alg.~\ref{alg:bh} (1st choice)&\(2.74\times 10^{-14}\)&\(1.04\times 10^{-14}\)&\(\hphantom{i}2.18\times 10^{-14}\)&\(7.99\times 10^{-15}\)\\
	\hspace{3.8pt}Alg.~\ref{alg:bh} (2nd choice)&\(2.77\times 10^{-14}\)&\(9.88\times 10^{-15}\)&\(\hphantom{i}1.80\times 10^{-14}\)&\(5.78\times 10^{-15}\)\\
	\hspace{2pt}Alg.~\ref{alg:bh} (3rd choice)&\(1.31\times 10^{-13}\)&\(5.22\times 10^{-14}\)&\(\hphantom{i}5.09\times 10^{-14}\)&\(1.76\times 10^{-14}\)\\
	BCGS2 (House)&\(4.40\times 10^1\hphantom{00}\)&\(4.90\times 10^{-13}\)&\(2.00\times 10^1\hphantom{0}\)&\(9.75\times 10^{-14}\)\\
	BCGS2 (Chol)&--&--&--&--\\
	\hline
\end{tabular}
\label{tab:orbhb}
\end{table}
\end{example}

In the following examples we illustrate that an ill-conditioned \(T\) in the
generalized Householder transformation can cause numerical instability.

\begin{example}
\label{ex:badmlu}
We follow the method in Appendix~\ref{sec:LU} to generate a matrix
\(V\in\mathbb R^{1000\times100}\) with orthonormal columns.
Applying the modified LU factorization on \(V_{1:100,:}\) yields the
ill-conditioned~\(U\) in~\eqref{eq:badmlu}.
The matrix \(A\in\mathbb{R}^{1000\times100}\) is generated randomly.
We apply Algorithm~\ref{alg:main} to \(V\) and \(A\) with different choices of
\(P\).
The results are collected in Table~\ref{tab:badmlu}.
The algorithm with the first choice loses accuracy due to the ill-conditioning
of \(U\), while the others are numerically stable.

\begin{table}[!tb]
\centering
\caption{Losses of orthogonality and relative residuals for
	Algorithm~\ref{alg:main} on Example~\ref{ex:badmlu}.}
\begin{tabular}{cccc}
	\hline
	Method&\(\norm{V\herm \hat{Q}}_2\)&\(\norm{\hat{Q}\herm \hat{Q}-I}_2\)&\(\norm{A-V\hat{S}-\hat{Q}\hat{R}}_2/\norm{A}_2\)\\
	\hline
	Alg.~\ref{alg:main} (1st choice)&\(8.29\times 10^{-11}\)&\(3.51\times 10^{-6\hphantom{1}}\)&\(6.45\times 10^{-6\hphantom{1}}\)\\
	\hspace{3.8pt}Alg.~\ref{alg:main} (2nd choice)&\(6.12\times 10^{-16}\)&\(1.21\times 10^{-15}\)&\(1.93\times 10^{-15}\)\\
	\hspace{2pt}Alg.~\ref{alg:main} (3rd choice)&\(5.68\times 10^{-16}\)&\(1.42\times 10^{-15}\)&\(1.94\times 10^{-15}\)\\
	\hline
\end{tabular}
\label{tab:badmlu}
\end{table}
\end{example}

\begin{example}
We generate various \(V\in\mathbb{R}^{1000\times 100}\) with orthonormal
columns and \(P\in\mathbb{R}^{100\times 100}\).
Instead of choosing \(P\) according to the three choices in
Section~\ref{subsec:choices}, in this example each pair of \((V,P)\) is
carefully constructed so that the matrix \(T\) in the generalized Householder
transformation has a prescribed condition number.
The matrix \(A\in\mathbb{R}^{1000\times100}\) is generated randomly.
The results of Algorithm~\ref{alg:main} are displayed in Figure~\ref{fig:cond}.
The growth of both losses of orthogonality and relative residuals matches the
predictions of our rounding error analysis in Section~\ref{sec:rounding}.

\begin{figure}[!tb]
\centering
\includegraphics[scale=0.5]{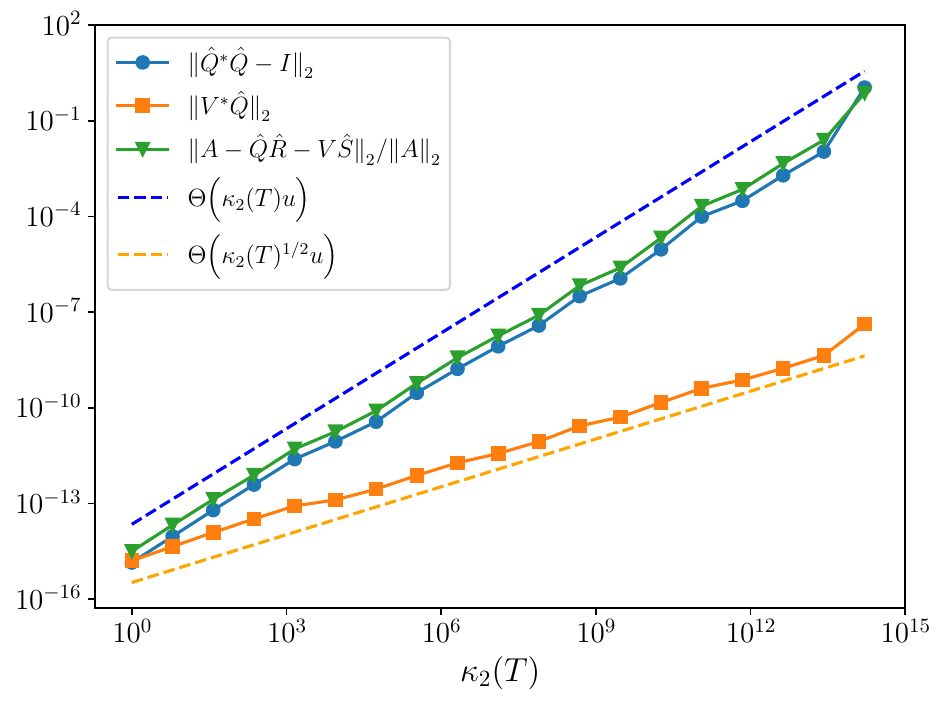}
\caption{Losses of orthogonality and relative residuals with different
	\(\kappa_2(T)\)'s.}
\label{fig:cond}
\end{figure}
\end{example}

\subsection{Performance tests}
We also implement our algorithms in Fortran~90 in order to test the actual
performance.
The code is compiled using the GNU Fortran compiler version 11.4.0 with
optimization flag \texttt{-O3}, and linked with OpenBLAS 0.3.26 and LAPACK
3.11.0.
The subroutine \texttt{xLA\{OR,UN\}HR\_COL\_GETRFNP} is
employed to perform the modified LU factorization.
Meanwhile, \texttt{xGEQRF} is used to conduct Householder-QR, and
\texttt{xGEQRFP} is utilized for the QR factorization, where the diagonal
elements of the matrix \(R\) are nonnegative.
We use \texttt{xGESDD} to compute the polar decomposition, although more
efficient approximation-based methods exist~\cite{NF2016}.

We generate \(V\in\mathbb{C}^{10000\times 100}\) and
\(A\in\mathbb{C}^{10000\times k}\) with prescribed condition number
\(\kappa_2(A)=10^{12}\) using the subroutine \texttt{xLAGGE}.
Excution times of Algorithm~\ref{alg:main} relative to the naive
Householder-QR algorithm are shown in Figure~\ref{fig:time}.
BCGS is excluded from the comparison, since the computed
\(\lVert V\herm \hat{Q}\rVert_2\) is far above \(\urf\),
suggesting numerical instability.
Figure~\ref{fig:time} indicates that Algorithm~\ref{alg:main} outperforms
naive Householder-QR and BCGS2 in computational efficiency,
particularly for the first two choices of \(P\).
Since Table~\ref{tab:badmlu} already reveals that the first choice can suffer
from numerical instability, in practice we recommend adopting the second
choice which is based on the QR factorization to generate \(P\) in
Algorithm~\ref{alg:main}.

\begin{figure}[!tb]
\centering
\includegraphics[scale=0.5]{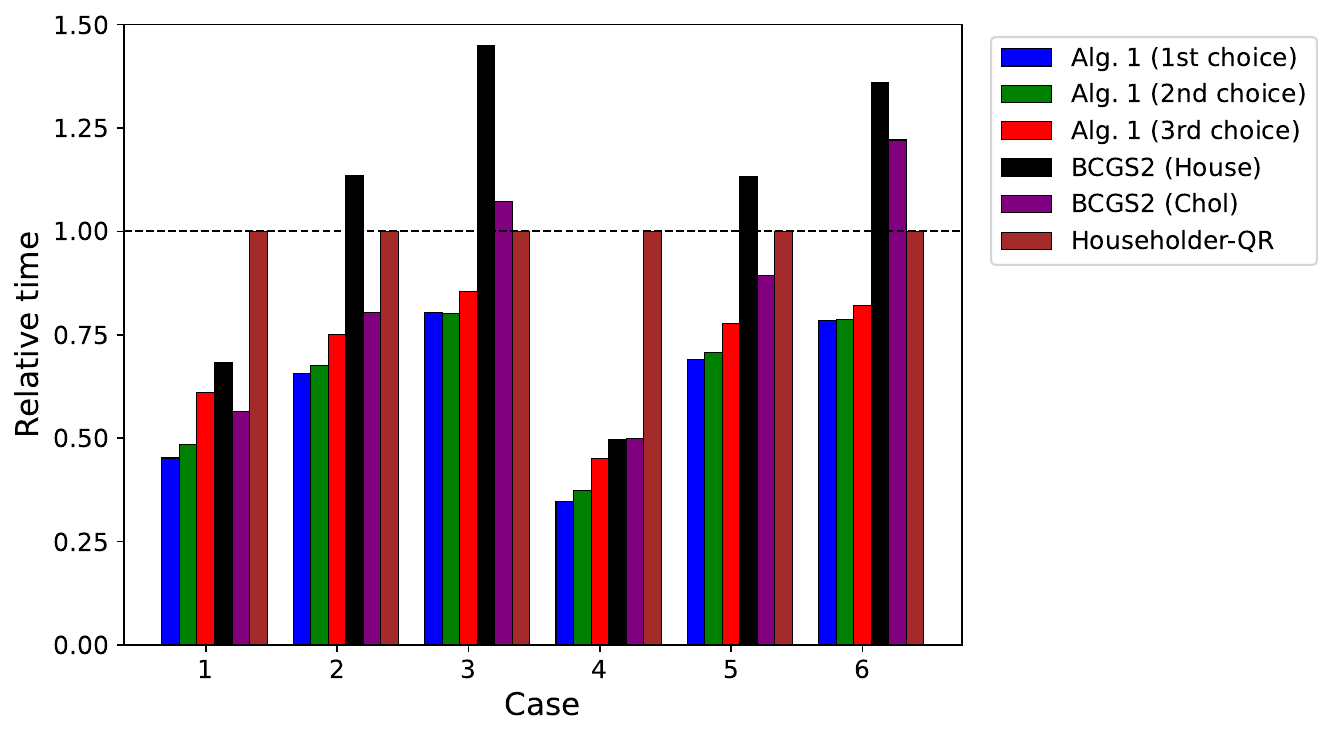}
\caption{Execution time relative to the Householder-QR performed on \([V,A]\).
Cases~1--3 correspond to real arithmetic, and Cases~4--6 to complex arithmetic.
In both settings \(k_0=100\) and \(k=50\), \(100\), \(200\).}
\label{fig:time}
\end{figure}

\section{Conclusion}
As a commonly used approach, BCGS may encounter instability when handling
two-stage orthogonalization.
To address this, we propose an unconditionally stable two-stage Householder
orthogonalization algorithm.
The key ingredient of our algorithm is to use the generalized Householder
transformation, which has been studied by~\cite{BJ1974,Kaufman1987,SP1988}.
We further discuss three variants (corresponding to three choices of the
matrix \(P\) as in Section~\ref{subsec:choices}) to enhance numerical
stability.
The second choice is recommended in practice as it is both cheap and
numerically stable.
We conduct a rounding error analysis to confirm the unconditional stability of
our algorithm.
Numerical experiments demonstrate that our algorithm is both numerically
stable and computationally efficient.

\section*{Acknowledgments}

We thank Laura Grigori, Daniel Kressner, and Yuxin Ma for helpful discussions.

\appendix
\addcontentsline{toc}{section}{Appendix}
\section*{Appendix}
\section{The matrix \(U\) in modified LU}
\label{sec:LU}

We begin with the following lemma.
\begin{lemma}
\label{lem:U}
Given a matrix \(V\in\mathbb{R}^{n\times k}\) with \(V\herm V=I_k\),
and a vector \(b\in\mathbb{R}^{k+1}\) with \(\lVert b\rVert_2<1\).
There exists \(\tilde{V}\in\mathbb{R}^{(n+1)\times (k+1)}\) with orthonormal columns such that
\[
e_1\herm\tilde{V}=b\herm,\quad
H\tilde{V}=\tilde{V}_H=\begin{bmatrix}
-\sign(b_1)&0\\
0&V
\end{bmatrix},
\]
where \(H\) is a Householder matrix
generated by the first step of Householder-QR applying to \(\tilde{V}\).
\end{lemma}

\begin{proof}
We can find a unit vector \(v\in\mathbb{R}^{n+1}\)
such that \(\tilde{V}_H\herm v=0\).
Let
\[
x=\sqrt{1-\lVert b\rVert_2^2}\,v+\tilde{V}_Hb,
\]
and then \(x_1=-\sign(b_1)b_1\) (because \(v_1=0\)).
Let \(\tau=1-x_1\),
\[
w=\frac{x-e_1}{x_1-1},\quad H=I-\tau ww\herm,
\]
and thus \(Hx=e_1\).
Let \(\tilde{V}=H\tilde{V}_H\) and
we can verify that \(\tilde{V}\) satisfies
\[
e_1\herm\tilde{V}= e_1\herm H \tilde{V}_H=x\herm\tilde{V}_H = b\herm.
\]
Note that the first column of \(\tilde{V}\) is
\(-\sign(b_1)He_1=-\sign(b_1)x\) and \(\tilde{V}_{1,1}=b_1\).
If we apply Householder-QR to \(\tilde{V}\),
the first Householder matrix is \(H\).
\end{proof}

Building upon Lemma~\ref{lem:U},
we can prove the following theorem,
thereby verifying that \eqref{eq:badmlu} is a feasible output
in the modified LU factorization.
\begin{theorem}
Suppose \(R\in\mathbb{R}^{k\times k}\) is upper triangular
and the \(2\)-norm of each row of \(R\) is less than one.
Let \(D\in\mathbb{R}^{k\times k}\) be a diagonal matrix that satisfies \(d_{i,i}=\sign(R_{i,i})\) for \(i=1\), $\dotsc$,~\(k\),
Then there exists a matrix \(Z\in\mathbb{R}^{k\times k}\), \(\lVert Z\rVert_2\le 1\) for which the modified LU factorization yields \(U=D+R\).
\end{theorem}

\begin{proof}
In modified LU, each row of \(U\) can be obtained from the first row of the Schur complement after each elimination step.
Suppose \(V\in\mathbb{R}^{n\times k}\) is a matrix with orthonormal columns.
It is shown in~\cite[Lemma 6.2]{BDGJKN2015} that the Schur complement update from modified LU applied to \(V\) (the whole matrix,
rather than \(V_{1:k_0,:}\)) matches the trailing matrix
update from Householder-QR applied to~\(V\) in every step.
Thus we can construct \(V\) by reversing the Householder-QR process
through~\(U\), or equivalently, through the matrix \(R\).
After~\(V\) is obtained, its first \(k\) rows \(V_{1:k,:}\) constitute the
desired matrix~\(Z\).

In the first step, we can find a vector \(x\in\mathbb{R}^{n-k}\)
such that \(V_1=[R_{k,k},x\herm]\herm\) is a unit vector.
The matrix \(V_1\) can be viewed as the trailing submatrix after applying
\(k-1\) Householder transformations.
Let~\(b^{(1)}=[R_{k-1,k-1},R_{k-1,k}]\herm\) and by Lemma~\ref{lem:U}
we can find a \(V_2\in\mathbb{R}^{(n-k+2)\times 2}\) with orthonormal columns.
In the \(i\)th step, we apply Lemma~\ref{lem:U} to \(V_i\) and
\(b^{(i)}=[R_{k-i,k-i},\dotsc,R_{k-i,k}]\herm\).
By induction after \(k-1\) steps we obtain \(V=V_k\) that satisfies the conditions.
\end{proof}

\section{Validity of assumptions in Section~\ref{sec:rounding}}
\label{sec:assump}

\subsection{Assumption~\ref{assump:qr}}
This assumption is automatically satisfied provided that
Householder-QR is used; see \cite[Chapter~19]{Higham2002}.

\subsection{Assumption~\ref{assump:T}}
If we use \(I_{k_0}-V_{1:k_0,:}\herm P\) to compute \(\hT\),
then rounding errors arise from matrix multiplication and addition.
As long as \(P\) is numerically unitary, i.e.,
\(\lVert P\herm P-I_{k_0}\rVert_2=O(\urf)\),
we have \(\lVert\hT-T\rVert_2=O(\urf)\) as well.
Next we examine the scenario in which \(\hT\) is obtained using the methods
described in Section~\ref{subsec:choices}.
For simplicity, we assume \(V\herm V=I_{k_0}\).

\paragraph{First choice}
If we use the modified LU factorization to construct \(P\),
we have \(\hT\) in factorized form~\((\hat{L}\hat{U})\herm P\),
where \(\lvert\hat{L}\hat{U}-P+V_{1:k_0,:}\rvert
\le\epsilon_{\mathrm{LU}}(k_0)
\lvert\hat{L}\rvert\lvert\hat{U}\rvert\).
The absolute value and the inequality of matrices
are understood componentwise and \(\epsilon_{\mathrm{LU}}(k_0)=O(k_0\urf)\)~\cite[Theorem 9.3]{Higham2002}.
Furthermore, by~\cite[Lemma 6.2]{BDGJKN2015},
we can verify \(\lVert L\rVert_1\le\sqrt{k_0}+1\),
\(\lVert U\rVert_\infty\le \sqrt{k_0}+1\). Then
\begin{align*}
\lVert\hT-T\rVert_2
&=\lVert(\hat{L}\hat{U})\herm P-(P-V_{1:k_0,:})\herm P\rVert_2\\
&\le\lVert\,\lvert \hat{L}\rvert\,\rVert_2\lVert\,\lvert \hat{U}\rvert\,
\rVert_2\epsilon_{\mathrm{LU}}(k_0)\\
&\lesssim k_0(\sqrt{k_0}+1)^2\epsilon_{\mathrm{LU}}(k_0).
\end{align*}
We use \(\lesssim\) here because we ignore the small error between \(L\), \(U\)
and \(\hat{L}\), \(\hat{U}\).

\paragraph{Second choice}
If we use the QR factorization, we have
\[
\lVert P\herm P-I_{k_0}\rVert_2\le\epsilon_{\mathrm{orth}}(k_0,k_0), \qquad
\lVert V_{1:k_0,:}+PR_1\rVert_2\le\epsilon_{\mathrm{res}}(k_0,k_0)
\]
by Assumption~\ref{assump:qr} in Section~\ref{sec:rounding}.
Then
\[
\hT=\fl(I_{k_0}+R_1\herm)=I_{k_0}+R_1\herm+\Delta T_1,
\qquad \abs{\Delta T_1}\le\urf\abs{I_{k_0}+R_1\herm}.
\]
Thus
\begin{align*}
\lVert\hT-T\rVert_2&=\lVert\Delta T_1+(V_{1:k_0,:}+PR_1)\herm P
+(R_1\herm+I_{k_0})(I_{k_0}-P\herm P)\rVert_2\\
&\le\lVert\Delta T_1\rVert_2+\norm{P}_2\norm{V_{1:k_0,:}+PR_1}_2
+(\lVert R_1\rVert_2+1)\lVert P\herm P-I_{k_0}\rVert_2\\
&\le(\sqrt{2}+1)\sqrt{k_0}\urf+\sqrt{2}\epsilon_{\mathrm{res}}(k_0,k_0)
+(\sqrt{2}+1)\epsilon_{\mathrm{orth}}(k_0,k_0).
\end{align*}

\paragraph{Third choice}
The analysis follows analogously to the second choice,
assuming the polar decomposition is computed
via a backward stable algorithm
such as the Golub--Kahan--Reinsch SVD~\cite[Section~5.4.1]{GV2013}.
Therefore we omit the proof.

\subsection{Assumption~\ref{assump:sv}}
The assumption holds when \(A\) is triangular or positive definite,
with \(\epsilon_{\mathrm{sv}}(n)=\sqrt{n}\,\gamma_n\)
or \(\gamma_{3n+1}n/(1-n\gamma_{n+1})\), respectively;
see~\cite[Chapters~8 and~10]{Higham2002}.
For a general dense matrix \(A\), if we solve \(Ax=b\) by LU factorization,
the backward error \(\Delta A\) is bounded in terms of the growth factor in
the factorization.
Fortunately, for modified LU, we have \(\lVert L\rVert_1\le\sqrt{n}+1\),
\(\lVert U\rVert_\infty\le \sqrt{n}+1\),
thus \(\epsilon_{\mathrm{sv}}(n)\lesssim n(\sqrt{n}+1)^2
(\epsilon_{\mathrm{LU}}(n)+\gamma_{2n})\);
see~\cite[Chapter~9]{Higham2002}.
Under our setting, we actually have
\(\epsilon_{\mathrm{sv}}(n)\lesssim n(\sqrt{n}+1)^2\gamma_{2n}\) since the
term \(\epsilon_{\mathrm{LU}}\) has already been absorbed in \(\epsilon_T\).
Therefore the choices in Section~\ref{subsec:choices} all yield small \(\epsilon_{\mathrm{sv}}(n)\).

\section{Proof of Theorem~\ref{thm:orth}}
\label{sec:proof}

We first establish the following lemma.

\begin{lemma}
\label{lem:Hx}
Under the assumption of Theorem~\ref{thm:orth},
assume \(x\in\mathbb{C}^{n}\).
There exists a matrix \(\tilde{W}\) which satisfies \(\lvert\tilde{W}-\hW\rvert\le\gamma_{k_0}\lvert\hW\rvert\)
and a matrix \(\tilde T\) which satisfies \(\lVert\tilde{T}-\hT\rVert_2\le\epsilon_{\mathrm{sv}}(k_0)\lVert\hT\rVert_2\) such that
\begin{align*}
\fl\big(x-\hW\hT^{-1}\hW\herm x\big)=\tilde{H} x+\Delta x_1,
\quad \tilde{H}=I_n-\tilde{W}\tilde{T}^{-1}\tilde{W}\herm,
\end{align*}
where \(\lVert\Delta x_1\rVert_2/\lVert x\rVert_2\le\epsilon_3\).
We also have
\[
\fl\big(x-\hW\hT^{-1}\hW\herm x\big)=H x+\Delta x_2,
\]
where \(\lVert\Delta x_2\rVert_2/\lVert x\rVert_2\le\epsilon_4\).
\end{lemma}

\begin{proof}
We have
\begin{align*}
\fl\big(\hT^{-1}\hW\herm x\big)&=
(\hT+\Delta T)^{-1}\big(\hW+\Delta W_1\big)\herm x,
\end{align*}
where \(\lvert\Delta W_1\rvert\le\gamma_n\big\lvert\hW\big\rvert\) and
\(\lVert\Delta T\rVert_2\le\epsilon_{\mathrm{sv}}(k_0)\lVert\hT\rVert_2\).
Let \(\tilde{T}=\hT+\Delta T\) and then
\begin{align*}
\fl\big(x-\hW\hT^{-1}\hW\herm x\big)=x-\big(\hW+\Delta W_2\big)\tilde{T}^{-1}\big(\hW+\Delta W_1\big)\herm x+\Delta x_3,
\end{align*}
where \(\lvert\Delta W_2\rvert\le\gamma_{k_0}\big\lvert\hW\big\rvert\) and
\[
\lvert\Delta x_3\rvert\le\urf\cdot\bigl\lvert x-\bigl(\hW+\Delta W_2\bigr)\tilde{T}^{-1}\bigl(\hW+\Delta W_1\bigr)\herm x\bigr\rvert.
\]
Let \(\tilde{W}=\hW+\Delta W_2\) and
we can verify that \(\fl\big(x-\hW\hT^{-1}\hW\herm x\big)=\tilde{H} x+\Delta x_1\), where
\begin{align*}
\Delta x_1=\Delta x_3+\tilde{W}\tilde{T}^{-1}(\Delta W_2-\Delta W_1)\herm x.
\end{align*}
Let \(\Delta x_4\) denote the second term of \(\Delta x_1\). Then
\[
\lVert\Delta x_4\rVert_2\le\sqrt{k_0}\gamma_{n+k_0}\lVert\tilde{W}\tilde{T}^{-1}\rVert_2\lVert\hW\rVert_2\lVert x\rVert_2.
\]
For \(\Delta x_3\), we have
\begin{align*}
\lvert\Delta x_3\rvert&\le\urf\lvert\tilde{H} x+\Delta x_4\rvert,\\
\lVert\Delta x_3\rVert_2&\le\urf(\lVert\tilde{H}\rVert_2\lVert x\rVert_2+\lVert\Delta x_4\rVert_2).
\end{align*}
Thus
\[
\lVert\Delta x_1\rVert_2\le\lVert\Delta x_3\rVert_2+\lVert\Delta x_4\rVert_2
\le\sqrt{k_0}\gamma_{n+k_0+1}\lVert\tilde{W}\tilde{T}^{-1}\rVert_2\lVert\hW\rVert_2\lVert x\rVert_2+\urf\lVert\tilde{H}\rVert_2\lVert x\rVert_2.
\]

It is easy to verify
\begin{align*}
\lVert T-\tilde{T}\rVert_2&\le\epsilon_T(k_0)+
\epsilon_{\mathrm{sv}}(k_0)\lVert\hT\rVert_2\\
&\le\epsilon_T(k_0)+
\epsilon_{\mathrm{sv}}(k_0)
\big(1+\lVert V\rVert_2\lVert P\rVert_2\big)(1+\epsilon_T(k_0))\\
&\le\epsilon_T(k_0)+3\epsilon_{\mathrm{sv}}(k_0)
(1+\epsilon_T(k_0)).
\end{align*}
Since \(\hW=[P\herm,0]\herm-V+\Delta W\) where \(\abs{\Delta W}\le\urf\abs{[P\herm,0]\herm-V}=\urf\abs{W}\),
we have
\[
\lvert W-\tilde{W}\rvert\le \big(\urf+(1+\urf)\gamma_{k_0}\big)\lvert W\rvert
\le\gamma_{k_0+1}\lvert W\rvert.
\]
Thus \(\lVert W-\tilde{W}\rVert_2\le\sqrt{k_0}\gamma_{k_0+1}\lVert W\rVert_2
\le2\sqrt{2k_0}\gamma_{k_0+1}\).
Note that \(\lVert\tilde{W}\herm\tilde{W}-\tilde{T}-\tilde{T}\herm\rVert_2
\le\epsilon_1\).
By Lemma~\ref{lem:WT}, \(\lVert\tilde{W}\tilde{T}^{-1}\rVert_2\le\eta\).
Using Theorem~\ref{thm:hH}, we obtain
\[
\lVert H-\tilde{H}\rVert_2\le(\delta_T+8k_0\gamma_{k_0+1}^2)\eta^2+
4\sqrt{2k_0}\gamma_{k_0+1}\eta=\epsilon_2.
\]
Therefore,
\begin{align*}
\frac{\lVert\Delta x_1\rVert_2}{\lVert x\rVert_2}&\le
\sqrt{k_0}\gamma_{n+k_0+1}\lVert\tilde{W}\tilde{T}^{-1}\rVert_2\lVert\hW\rVert_2+\urf\lVert\tilde{H}\rVert_2\\
&\le2(1+\sqrt{k_0}\urf)\sqrt{2k_0}\gamma_{n+k_0+1}\cdot\eta+\urf(\lVert H\rVert_2+\epsilon_2)\\
&\le2\sqrt{2k_0}\gamma_{n+2k_0+1}\cdot\eta+\urf(\sqrt{1+\delta \eta^2}+\epsilon_2)=\epsilon_3.
\end{align*}
Let \(\Delta x_2=(\tilde{H}-H)x+\Delta x_1\) and then
\[
\frac{\lVert\Delta x_2\rVert_2}{\lVert x\rVert_2}
\le\norm{\tilde{H}-H}+\frac{\lVert\Delta x_1\rVert_2}{\lVert x\rVert_2}\le\epsilon_2+\epsilon_3=\epsilon_4.\qedhere
\]
\end{proof}

With the help of Lemma~\ref{lem:Hx}, we are ready to prove
Theorem~\ref{thm:orth}.

\begin{proof}[Proof of Theorem~\ref{thm:orth}]
By Lemma~\ref{lem:Hx}, we have
\[
\hat{Q}=H\begin{bmatrix}0\\
~\underline{\hat Q}~\end{bmatrix}+\Delta Q,\quad \lVert\Delta Q\rVert_{\fro}\le\lVert\underline{\hat Q}\rVert_{\fro}\epsilon_4.
\]
It is easy to verify
\begin{align*}
\lVert\hat{Q}\herm \hat{Q}-I\rVert_2&\le
\epsilon_{\mathrm{orth}}(n-k_0,k)
+\lVert\underline{\hat Q}\rVert_2^2\lVert H\herm H-I_n\rVert_2
+2\lVert H\rVert_2\lVert\underline{\hat Q}\rVert_2\lVert\underline{\hat Q}\rVert_\fro\epsilon_{4}+\lVert\underline{\hat Q}\rVert_\fro^2\epsilon_4^2\\
&\le\epsilon_{\mathrm{orth}}(n-k_0,k)
+\lVert\underline{\hat Q}\rVert_2^2
(\delta \eta^2+2\sqrt{k}\lVert H\rVert_2\epsilon_4+k\epsilon_4^2)\\
&\le\epsilon_{\mathrm{orth}}(n-k_0,k)
+2(\delta \eta^2+2\sqrt{k(1+\delta \eta^2)}\epsilon_4+k\epsilon_4^2).
\end{align*}
Let \(\hat{q}_i\) and \(\underline{\hat{q}}_i\) denote the \(i\)th
column of \(\hat{Q}\) and \(\underline{\hat{Q}}\), respectively.
Then according to Lemma~\ref{lem:Hx}
\[
\hat{q}_i=\Big(\tilde{H}_i\begin{bmatrix}0\\
~\underline{\hat{q}}_i~\end{bmatrix}+\Delta q_i\Big),
\qquad \tilde{H}_i=I-\tilde{W}_i\tilde{T}_i^{-1}\tilde{W}_i\herm,
\qquad \lVert\Delta q_i\rVert_2/\lVert \underline{\hat{q}}_i\rVert_2\le\epsilon_3,
\]
where \(\tilde{W}_i\) and \(\tilde{T}_i\)
are the perturbed counterparts of the matrices \(W\) and \(T\).
By Theorem~\ref{thm:hH}, we obtain
\[
\lVert V\herm\tilde{H}_i-[P\herm,0]\herm\rVert_2\le
2\sqrt{2k_0}\gamma_{k_0+1}+
(\delta+\delta_T+2\sqrt{2k_0}
\gamma_{k_0+1}\lVert V\rVert_2)\eta.
\]
Thus
\begin{align*}
\frac{\lVert V\herm \hat{q}_i\lVert_2}{\lVert\underline{\hat{q}}_i\rVert_2}&\le
2\sqrt{2k_0}\gamma_{k_0+1}+
(\delta+\delta_T+4\sqrt{k_0}\gamma_{k_0+1})\eta+\lVert V\rVert_2\epsilon_3\\
&\le2\sqrt{2k_0}\gamma_{k_0+1}+
\epsilon_1\eta+\sqrt{2}\epsilon_3.
\end{align*}
Therefore,
\begin{align*}
\lVert V\herm \hat{Q}\rVert_\fro
\le\sqrt{2k}\,(2\sqrt{2k_0}\gamma_{k_0+1}+
\epsilon_1\eta+\sqrt{2}\epsilon_3).
\end{align*}
To analysis the residual we have
\[
\fl(A-WT\iherm W\herm A)= H\herm A+\Delta A,
\qquad \lVert\Delta A\rVert_\fro\le\lVert A\rVert_\fro\epsilon_4.
\]
By Theorem~\ref{thm:H}, \(H[P\herm,0]\herm=V\)
and then \(\lVert H_{:,1:k_0}-VP\herm\rVert_2\le\delta_2\lVert H\rVert_2\).
Let \(\tilde{A}=H\herm A+\Delta A\).
We have
\[
\hat{S}=P\herm \tilde{A}_{1:k_0,:}+\Delta S, \quad
\lvert\Delta S\rvert\le\gamma_{k_0}\lvert P\herm\rvert
\lvert\tilde{A}_{1:k_0,:}\rvert.
\]
Note that
\begin{align*}
A-\hat{Q}\hat{R}-V\hat{S}=A-H\begin{bmatrix}0\\~\underline{\hat Q}~\end{bmatrix}\hat{R}+\Delta Q\hat{R}-V(P\herm\tilde{A}_{1:k_0,:}+\Delta S).
\end{align*}
We have
\begin{align*}
\Big\lVert A-H\begin{bmatrix}0\\~\underline{\hat Q}~\end{bmatrix}\hat{R}-VP\herm\tilde{A}_{1:k_0,:}\Big\rVert_2&\le
\norm{A-H\tilde{A}}_2+
\Big\lVert H\tilde{A}-H\begin{bmatrix}0\\~\underline{\hat Q}~\end{bmatrix}\hat{R}-H_{:,1:k_0}\tilde{A}_{1:k_0,:}\Big\rVert_2\\
&\hspace{12pt}+\lVert(H_{:,1:k_0}-VP\herm)\tilde{A}_{1:k_0,:}\rVert_2\\
&\le\delta\lVert A\rVert_2 \eta^2+\lVert H\Delta A\rVert_2
+\lVert H\rVert_2\lVert \tilde{A}\rVert_2
(\epsilon_{\mathrm{res}}(n-k_0,k)+\delta_2).
\end{align*}
Thus
\begin{align*}
\frac{\lVert A-\hat{Q}\hat{R}-V\hat{S}\lVert_2}{\lVert A\rVert_2}&\le
\delta \eta^2+\sqrt{k}\lVert H\rVert_2\epsilon_4
+\norm{H}_2(\norm{H}_2+\sqrt{k}\epsilon_4)
(\epsilon_{\mathrm{res}}(n-k_0,k)+\delta_2)\\
&\hspace{12pt}+\frac{\lVert\Delta Q\rVert_2\lVert\hat{R}\rVert_2+\lVert V\rVert_2\lVert\Delta S\rVert_2}{\lVert A\rVert_2}\\
&\le\delta \eta^2+\sqrt{k}\lVert H\rVert_2\epsilon_4
+\norm{H}_2(\norm{H}_2+\sqrt{k}\epsilon_4)
(\epsilon_{\mathrm{res}}(n-k_0,k)+\delta_2)\\
&\hspace{12pt}+(2\sqrt{k}\epsilon_4+2k_0\gamma_{k_0})
(\norm{H}_2+\sqrt{k}\epsilon_4)\\
&\le\delta \eta^2+(\sqrt{1+\delta \eta^2}+\sqrt{k}\epsilon_4)^2
(\epsilon_{\mathrm{res}}(n-k_0,k)+\delta_2+3\sqrt{k}\epsilon_4+2k_0\gamma_{k_0}),
\end{align*}
which completes the proof.
\end{proof}

\phantomsection

\end{document}